\documentclass{llncs}
\usepackage{llncsdoc}

\usepackage{amsmath, amssymb}
\usepackage{algorithm}
\usepackage{algorithmic}
\usepackage[normalem]{ulem}
\usepackage[switch, mathlines, displaymath]{lineno}
\linenumbers
\usepackage{times}
\usepackage{algorithm}
\usepackage{url, color, epsfig, amsmath, amssymb}
\usepackage[english]{babel}
\usepackage{graphicx}

\floatname{algorithm}{Procedure}

\newtheorem{thm}{Theorem}
\newtheorem{obs}[thm]{Observation}

\title{Network Decontamination with a Single Agent}

\author{Yessine Daadaa\inst{1}
\and Asif Jamshed\inst{2}
\and Mudassir Shabbir\inst{2}}

\institute{ School of Electrical Engineering and Computer Science,\\
Faculty of Engineering,\\
University of Ottawa,\\
Ottawa, Canada\\
\and Department of Computer Science, Rutgers University, NJ, USA\\
\email{ydaadaa@site.uottawa.ca},\\
\email{\{ajamshed,mudassir\}@cs.rutgers.edu}\\}

\begin{document}

\maketitle

\begin{abstract}
Faults and viruses often spread in networked environments by propagating from site to neighboring site. We model this process of
 {\em network contamination} by graphs. Consider a graph $G=(V,E)$, whose vertex set is contaminated and our goal is to decontaminate
 the set $V(G)$ using mobile decontamination agents that traverse along the edge set of $G$. Temporal immunity $\tau(G) \ge 0$ is 
defined as the time that a decontaminated vertex of $G$ can remain continuously exposed to some contaminated neighbor without getting
 infected itself. The \emph{immunity number} of $G$, $\iota_k(G)$, is the least $\tau$ that is required to decontaminate $G$ using $k$
 agents. We study immunity number for some classes of graphs corresponding to network topologies and present upper bounds on
 $\iota_1(G)$, in some cases with matching lower bounds. Variations of this problem have been extensively studied in literature,
 but proposed algorithms have been restricted to {\em monotone} strategies, where a vertex, once decontaminated, may not be 
recontaminated. We exploit nonmonotonicity to give bounds which are strictly better than those derived using monotone strategies.
\end{abstract}


\section{Introduction}

Faults and viruses often spread in networked environments by propagating from site to neighboring site.
 The process is called {\em network contamination}. Once contaminated, a network node might behave incorrectly, and it 
 could cause its neighboring node to become contaminated as well, thus propagating faulty computations. The propagation 
patterns of faults can follow different dynamics, depending on the behavior of the affected node, and topology of the network. 
At one extreme we have a full spread behavior: when a site is affected by a virus or any other malfunction, such a malfunction
 can propagate to all its neighbors; other times, faults propagate only to sites that are susceptible to be affected; the 
definition of susceptibility depends on the application but oftentimes it is based on local conditions, for example, a node 
could be vulnerable to contamination if a majority of its neighbors are faulty, and immune otherwise (e.g., see 
\cite{kutten1999fault}, \cite{kutten2000tight}, \cite{luccio2006network}); or it could be immune to 
contamination for a certain amount of time after being repaired (e.g., see \cite{daadaa2010network},
 \cite{flocchini2008tree}). 

In this paper we consider a propagation of faults based on what we call {\em temporal immunity}: a clean node 
is allowed to be exposed to contaminated nodes for a predefined amount of time after which it becomes {\em contaminated}.
 Actual decontamination is performed by mobile cleaning agents which which move from host to host over network connections.

\subsection{Previous Work}
\paragraph{\bf Graph Search.}
The decontamination problem considered in this paper is a variation of a problem extensively studied in the literature known as 
{\em graph search}.
The graph search problem was first introduced by Breish in \cite{breisch1967intuitive}, where an approach for the problem of finding 
an explorer that is lost in a complicated system of dark caves is given. Parsons (\cite{parson1976pursuit}\cite{parsons1978search}) 
proposed and studied
 the pursuit-evasion problem on graphs. Members of a team of searchers traverse the edges of a graph in pursuit of a fugitive,
 who moves along the edges of the graph with complete knowledge of the locations of the pursuers. 
The efficiency of a graph search solution is based on the size of the search team. Size of smallest search team that
can {\em clear} a graph $G$ is called search number, and is denoted 
in literature by $s(G)$. In \cite{megiddo1988complexity}, Megiddo et al. 
approached the algorithmic question: {\em Given an arbitrary $G$, how should one calculate $s(G)$?}
They proved that for arbitrary graphs, determining if the search number  is less than or
 equal to an integer $k$ is NP-Hard. They also gave algorithms to compute $s(G)$ where
 $G$ is a special case of trees. For their results, they use the fact that recontamination of a cleared
vertex does not help reduce $s(G)$, which was proved by LaPaugh in \cite{lapaugh1993recontamination}.
A search plan for  $G$ that does not involve recontamination of cleared vertices is referred to as a {\em monotone} plan.

\paragraph{\bf Decontamination.}
The model for decontamination studied in literature is defined as follows.
 A team of agents is initially located at the same node, the homebase, and all the other nodes are contaminated.
 A decontamination strategy consists of a sequence of movements of the agents along the edges of the network.
 At any point in time each node of the network can
 be in one of three possible states: clean, contaminated, or guarded. A node is guarded when it contains at least one 
agent. A node is clean when an agent passes by it and all its neighboring nodes are clean or guarded, contaminated otherwise.
 Initially all nodes are contaminated except for the homebase (which is guarded). The solution to the problem is given by devising
 a strategy for the agents to move in the network in such a way that at the end all the nodes are clean.\\
The tree was the first topology to be investigated. In \cite{barriere2002capture}, Barri\`ere et al. showed that 
for a given tree $T$, the minimum number of agents needed to decontaminate $T$ depends on the location of the 
homebase. They gave first strategies to decontaminate trees.\\
In \cite{flocchini2005size}, Flocchini et al. consider the problem of decontaminating a mesh graph. 
They present some lower bounds on the number of agents, number of moves, and time required
 to decontaminate a $p\times q$ mesh $(p\leq q)$. At least $p$ agents, $pq$ moves, and $p+q-2$ time 
units are required to solve the decontamination problem.
Decontamination in graphs with {\em temporal immunity}, which is similar to the model of decontamination 
used in this paper, was first introduced in \cite{flocchini2008tree} where minimum team size necessary to disinfect 
a given tree with temporal immunity $\tau$ is derived. The main difference between the classical 
decontamination model, and the new model in \cite{flocchini2008tree} is that, once an agent departs, 
the decontaminated node is immune for  a certain $\tau \geq 0$ (i.e. $\tau=0$ corresponds 
to the classical model studied in the previous work) time units to viral attacks from infected neighbors.
 After the temporal immunity $\tau$ has elapsed, recontamination can occur.\\
Some further work in the same model was done in \cite{daadaa2011decontamination},
where a two dimensional lattice is considered.

\subsection{Definitions and Terminology}
We only deal with connected finite graphs without loops or multiple edges. For a graph $G=(V,E),$  and 
a vertex $v \in V$ let $N(v),$ the neighborhood of $v,$ be the set of all vertices $w$ such that $v$ is 
connected to $w$ by an edge. Let $deg(v)$ denote the degree of a vertex $v$ which is defined to be the size of its neighborhood. 
Maximum and minimum degree of any vertex in $G$ is denoted by $\Delta(G)$ and $\delta(G)$ respectively. 
The shortest distance between any two vertices $u,v \in V$ is denoted by $dist(u,v)$ and eccentricity of $v \in V$ is 
the maximum $dist(u,v)$ for any other vertex $u$ in $G$. The radius of a graph, $rad(G)$, is the minimum eccentricity 
of any vertex of $G$ and the vertices whose eccentricity is equal to $rad(G)$ are called the \emph{center vertices}. 
The diameter of a graph, $diam(G)$, is the maximum eccentricity of a vertex in $G$.

$K_n$ is the complete graph on $n$ vertices. $K_{m,n}$ denotes the complete bipartite 
graph where the size of two partitions is $m$ and $n$. An acyclic graph is known as a tree and a vertex 
of degree $1$ in a tree is known as a \emph{leaf} of the tree. Rest of the tree terminology used is standard.
 A star graph, $S_n$, is a tree on $n+1$ vertices where one vertex has degree $n$ and the rest of the vertices are
 leaves. Sometimes a single vertex of a tree is labeled as the \emph{root} of the tree. In this case the 
tree is known as a {\em rooted} tree. If we remove the root vertex from a rooted tree it decomposes into one or more
subtrees; each such subtree along with the root is called a {\em branch}, denoted by $B_i$, of original tree.
 Similarly, an {\em arm} is the set of vertices that lie on the path from root to a leaf, denoted by $A_i$.

Other classes of graphs will be defined as and when needed.

\subsection{Decontamination Model Specification}
Our decontamination model is a synchronous system. We assume that initially, 
at time $t = 0,$ all vertices in the graph are contaminated. A \emph{decontaminating agent} (henceforth referred to as an agent)
 is an entity, or a marker, that can be placed on any vertex. A concept similar to this is referred to in the literature as a
 \emph{pebble} \cite{chung1989pebbling}. 
Assume that at some time step $k$, agent is at $v\in V$, then at the next time step, we may move the agent 
to any of the neighbors of $v$.
 Vertices visited in this process are marked {\em decontaminated}, 
or {\em disinfected}. Any vertex that the agent is currently
 placed on is considered to be decontaminated.

A decontaminated vertex can get contaminated by uninterrupted exposure, for a certain amount of time, 
to a contaminated vertex 
in its neighborhood.
For decontaminated $v$ if there is no agent placed on $v$ but some neighbor of $v$ is contaminated, we say that 
$v$ is \emph{exposed}. For a decontaminated vertex $v$ we define the \emph{exposure time} of $v$, $\Xi(v)$, 
as the duration time $v$ has been exposed. Every time an agent visits $v$, or all vertices in $N(v)$ are
decontaminated, we reset $\Xi(v)=0$.
We say that $G$ has temporal immunity $\tau(G)$ if a decontaminated vertex $v\in V$ can only be recontaminated if for 
uninterrupted $\tau(G)$ time units, there is a neighbor of $v$ (not necessarily unique) that is contaminated and an 
agent does not visit $v$ during that time period. 
Note that for any decontaminated vertex $v$ we have that $0 \le \Xi(v) \le \tau(G) - 1$.

\begin{figure}[ht]
\begin{center}
\includegraphics[height=2.1in]{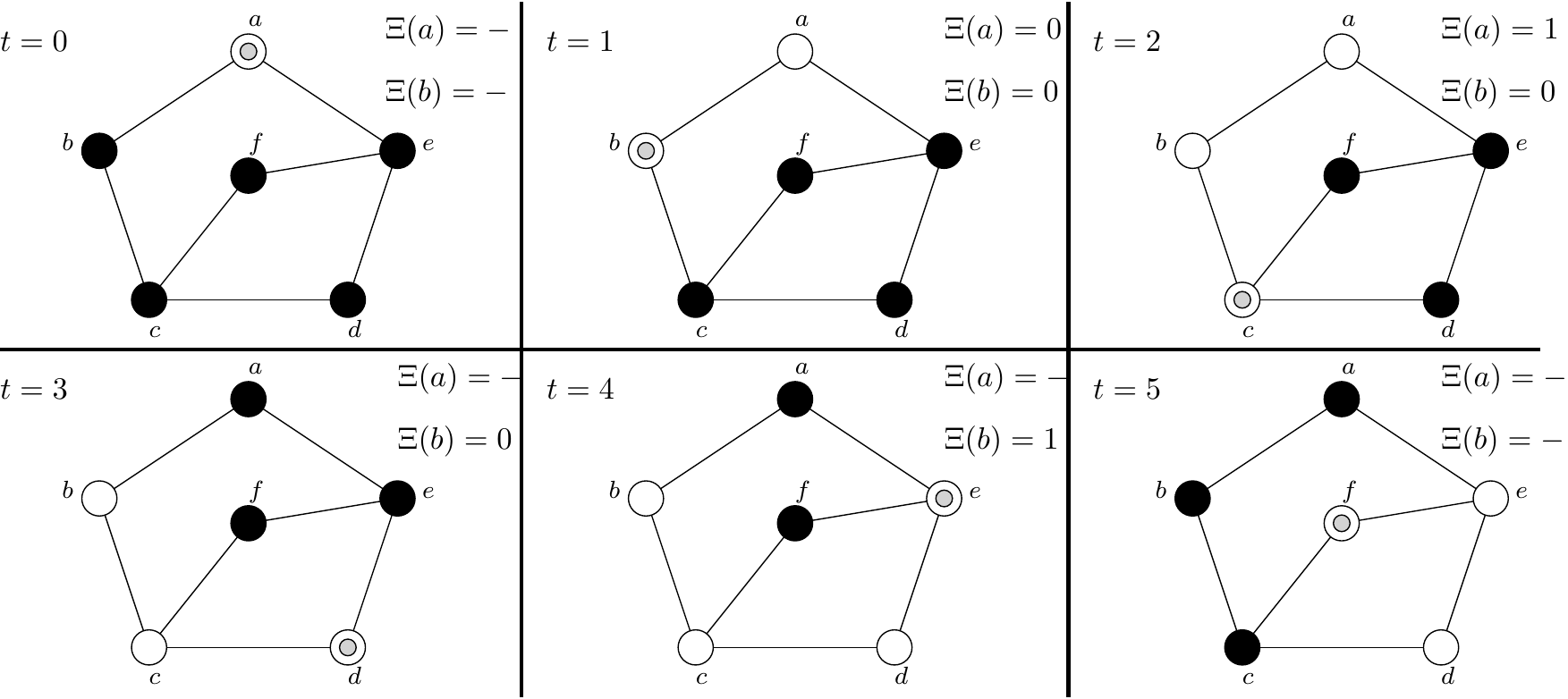}
\caption{Figure illustrates variation in exposure times of vertices $a$ and $b$ at different time steps as the agent tries to decontaminate $G$, with $\tau(G) = 2$.}
\label{example}
\end{center}
\end{figure}

Given a graph $G$, a temporal immunity $\tau$ and $n$ agents, our goal is to devise a decontamination strategy,
 which consists of choosing an initial placement for the agents and their movement pattern so that we can reach a
 state where all the vertices of $G$ are simultaneously decontaminated and we call the graph fully decontaminated. 
A strategy is called monotone if a decontaminated vertex is never recontaminated and is called nonmonotone otherwise. 
\emph{Immunity number} 
of $G$ with $k$ agents, $\iota_k(G)$, is the least $\tau$ for which full decontamination of $G$ is possible.
It is trivial to see that $\iota_k(G)$ is always finite for
$k\geq 1$.
\begin{obs}
Let $G$ be a connected graph on $n$ vertices, then  $\iota_k(G) \le 2(n - 1)$ for all $k\geq 1$.
\end{obs}
Without loss of generality, assume that $k=1$. First compute a spanning tree $T$ of $G$ and then make your 
the agent move in a depth first search order on $T$.
 Since the entire traversal takes exactly $2(n-1)$ steps this monotone strategy fully decontaminates any given graph.\\ 
However, in this paper we focus on decontamination of graph with a single agent; this gives us the liberty to use shortened
notation $\iota(G)$, and just $\iota$ when the graph is obvious from context, to mean $\iota_1(G)$, the immunity number of a graph
using a single agent.
\subsection{Our Results}
In section~\ref{section2} we prove bounds on $\iota$ for some simple graphs. In section~\ref{secSpider} we give
asymptotically sharp upper and lower bounds on $\iota(G)$ where $G$ is a mesh graph. We also give algorithms to decontaminate
spider graphs and $k$-ary trees. We then extend these techniques to give upper bound on immunity
number of general trees. In section~\ref{secDiscussion} we discuss some open problems. Our results are outlined in the table below.
\begin{table}
\begin{center}
\begin{tabular}[c]{|l|l|l|}
    \hline
    \centering{{\bf Graph Topology}}               & {\bf Upper Bound on $\iota$}                                 & {\bf Lower Bound on $\iota$}                                         \\ \hline
    Path $P_n$                         & $0$\hfill [Proposition~\ref{propPath}]                 & $0$ \hfill  [Proposition~\ref{propPath}]                       \\ \hline
    Cycle $C_n$                        & $2$\hfill   [Proposition~\ref{propCycle}]              & $2$\hfill   [Proposition~\ref{propCycle}]                      \\ \hline
    Complete Graph $K_n$               & $n-1$\hfill  [Theorem~\ref{thmComplete}]               & $n-1$\hfill  [Theorem~\ref{thmComplete}]                       \\ \hline
    Complete Bipartite Graph $K_{m,n}$, with $m\leq n$ & $2(m-1)$\hfill [Theorem~\ref{thmCompleteBipartite}]    & $2(m-1)$\hfill [Theorem~\ref{thmCompleteBipartite}]            \\ \hline
    Spider Graph on $n+1$ vertices     & $4\sqrt{n}$\hfill  [Corollary~\ref{corSpider}]         & - \\ \hline
    Tree on $n$ vertices                       & $30\sqrt{n}$\hfill  [Theorem~\ref{thmTreesUpperBound}] & - \\ \hline
    Mesh $m\times n$                   & ${m}$\hfill  [Theorem~\ref{meshUpperbound}]            & $\frac{m}{2}$ \hfill [Theorem~\ref{thmMeshLowerbound}]                                            \\ \hline
    Planar Graph on $n$ vertices       & $n-1$\hfill  [Theorem~\ref{thmGraphUpperBound}]        & $\Omega(\sqrt{n})$ \hfill [Corollary~\ref{corPlanarLowerBound}]                                          \\ \hline
    General Graphs                     & $n-1$\hfill  [Theorem~\ref{thmGraphUpperBound}]        & $n-1$\hfill   [Theorem~\ref{thmComplete}]                      \\ \hline
    \end{tabular}
    \caption{A summary of our results.}
\end{center}
    \end{table}
\section{Some Simple Graphs}
\label{section2}
We begin with the simple case when the graph that we want to decontaminate is a path.
\begin{proposition}
\label{propPath}
Let $P_n$ be a path on $n$ vertices, then $\iota(P_n) = 0$ for all $n \ge 1$.
\end{proposition}
It is easy to see that we do not need any temporal immunity to decontaminate the entire path if we start with our agent 
at one leaf vertex and at each time step we move it towards the other end until we reach it at $ t = n - 1$.

A cycle can be decontaminated using a similar strategy.
\begin{proposition}
Let $C_n$ be a cycle on $n$ vertices, $\iota(C_n) = 2$ for all $n \ge 4$.
\label{propCycle}
\end{proposition}
\begin{proof}
To see that $\iota(C_n) \le 2$ set the temporal immunity $\tau = 2$ and begin with the agent at any vertex of the cycle. 
At $t = 1$ choose one of it neighbors to move to. Henceforth, for $t = k \ge 2$, we always move our agent in a fixed, say clockwise,
direction.
It is straightforward to verify that we will end up with a fully decontaminated graph in at most $2n$ time steps. 
Note that this is a nonmonotone strategy. 


If we set temporal immunity $\tau = 1$ then we will show that we can never decontaminate more than two (adjacent) 
vertices of the cycle. Suppose that four vertices $v_n, v_1, v_2, v_3$ appear in the cycle in that order. 
Assume that at time step $t = 0$ the agent is placed at $v_1$ and, without loss of generality, it moves to $v_2$ at the 
next time step. At $t = 2$ if the agent moves to $v_3$ then $v_1$ becomes contaminated due
 to its exposure to $v_n$ and we end up with only $v_2$ and $v_3$ decontaminated which is the same as not having made 
any progress. If, on the other hand, the agent had moved back to $v_1$ at $t=2$ we would again have ended up with no 
progress since the agent would still have the same constraints on proceeding to its next vertex, therefore $\iota(C_n) > 1$.
\qed \end{proof}
\begin{remark}
Bound presented in Proposition~\ref{propCycle} is only tight because of our definition of $\tau$ as an integer. Otherwise
we observe that there always exists a strategy to decontaminate $C_n$ with $\tau=1+\epsilon$ for any real
number $\epsilon>0$ in finite time;
in fact same strategy as outlined in proof of upper bound above will work.
\end{remark}
\subsubsection{Complete Graph and Complete Bipartite Graph}
Path and cycle happen to be the simplest possible graphs that can be decontaminated easily with optimal constant immunity numbers
as seen above.
We now tend to some dense graphs and show that they may require much larger value of $\tau$.
\begin{thm}
\label{thmComplete}
Let $K_n$ be a complete graph on $n$ vertices, then $\iota(K_n) = n - 1$ for all $n \ge 4$.
\end{thm}
\begin{proof}
Let the vertex set $V = \{v_1, v_2, \ldots, v_n\}$. Since there is fully connected, we can fully decontaminate $K_n$ by 
making the agent visit all the vertices sequentially 
in any order giving us $\iota(K_n) \le n - 1$.\\

To see that this bound is actually tight we need to show that temporal immunity 
of $n - 2$ is not good enough for full decontamination. For this purpose set $\tau = n - 2$ and 
suppose that at time step $t = k$ we have somehow managed to decontaminate all the vertices of $K_n$ except 
one last vertex, say, $v_n$.
 Assume without loss of generality that the agent is at $v_{n-1}$. As long as the complete graph 
is not fully decontaminated, all the vertices which do not have the agent placed on 
them are exposed. This implies that the vertices $v_1, \ldots, v_{n-2}$ have all been visited by 
the agent in the last $n - 2$ time steps, that is, $\Xi(v_i) < n - 2$ for $1 \le i \le n - 2$. It also implies that since 
there is one agent, all these vertices
 have different exposure times, meaning that there is one vertex, say $v_1$, such that $\Xi(v_1) = n - 3$. At time step
 $k + 1$, if the agent moves to $v_n$ and decontaminates it, then $v_1$ becomes contaminated hence we make no progress; 
there is still one contaminated vertex remaining in the graph. If on the other hand agent $x$ is moved to $v_1$ to avoid 
its contamination, we will again have not made any progress. Moving the agent to any other vertex at $t = k+1$ actually increases 
the number of contaminated vertices in the graph.
\qed \end{proof}
The immunity number of complete bipartite graph depends upon the size of smaller partition.
\begin{thm}
\label{thmCompleteBipartite}
Let $G$ be a complete bipartite graph on the vertex sets $A$ and $B$ 
where $|A| = m, |B| = n$ such that $3 \leq m \leq n$, then $\iota(G) = 2m - 1$.
\end{thm}
\begin{proof}

Let $A = \{a_1, a_2, \ldots, a_m\}$ and $B = \{b_1, b_2, \ldots, b_n\}$. Set the temporal immunity $\tau = 2m - 1$ and place an agent 
at $a_1$ at $t=0$. Now we cycle through the vertices in $A$ and $B$ in an interleaved sequence as follows:
$$
a_1, b_1, a_2, b_2, a_3, b_2, \ldots, a_m, b_m, a_1, b_{m+1}, a_2, b_{m+2}, \ldots, b_n.
$$

When $t < 2m$ none of the vertices are exposed long enough to be recontaminated. At $t = 2m$ the agent returns to $a_1$, and thereafter
none of the decontaminated vertices in $B$ remain exposed while the vertices of $A$ keep getting visited by the agent before their
exposure time reaches $\tau$. It follows that this monotone strategy fully decontaminates $G$ in $2n -1$ time steps.

Our claim is that if $\tau < 2m - 1$ then it is not possible to fully decontaminate a partition during any stage of a given 
decontamination strategy. Consider a strategy that aims to fully decontaminate $A$ at some point (and $B$ is never fully decontaminated
before that).  Suppose that at time $t = k$ there remains exactly one contaminated vertex in $A$ (and that there were two contaminated vertices
in $A$ at $t = k - 1$). Note that this implies that the agent is at some vertex in $A$ at $t = k - 1$. Since $B$ has never 
fully been decontaminated, it follows that there exists a vertex $a_j \in A$ such that $\Xi(a_j) = 2m - 3$. 
Since it is a bipartite graph, it will take at least two additional time steps to reach the last contaminated 
vertex of $A$, and if the temporal immunity is less that $2m - 1$ the agent will fail to decontaminate $A$ fully.

In the case where the decontamination strategy requires that $B$ is fully decontaminated before 
$A$, similar argument gives us a lower bound of $2n - 1$ on $\iota(G)$ 
but we have already given a strategy that decontaminates $A$ first which gives a better upper bound.
\qed \end{proof}

\section{Spider, $k$-ary Tree, and Mesh Graph}
\label{secSpider}
Star and mesh are two important network topologies, which are extreme examples of centralization and decentralization respectively. In
 the following we study our problem on star, {\em spider} a generalization of star, $k$-ary trees, and mesh graphs. Some of the ideas and
proof techniques developed in this section will feature in proof of upper bound on immunity number for general trees in the next section.
\subsection{Spider}
Let $S$ be a star graph. The simple strategy of starting the agent at the center vertex and visiting each leaf in turn
 (via the center) gives us the best possible bound on $\iota(S_n)$.
\begin{lemma}
Temporal immunity $\tau = 1$ is necessary and sufficient for any star graph.
\end{lemma}
\begin{proof}
The strategy outlined above gives us the upper bound of $\iota(S_n) \le 1$. Argument for matching lower bound is 
straightforward and we omit the details. 
\qed \end{proof}

A graph that is structurally closely related to a star graph is the \emph{spider}. A spider is a tree in which one vertex, called 
the \emph{root}, has degree at least $3$, and all the rest of the vertices have degree at most $2$. Another way to look at it is that a
 spider consists of $k$ vertex disjoint paths all of whose one endpoint is connected to a root vertex. Such a spider 
is said to have $k$ \emph{arms}.

Let $S$ be a spider such that the degree of the root is $\Delta$. If $m$ is the length of the longest arm of $S$ then using a naive 
monotone strategy of visiting each arm sequentially, starting at the root and traversing each arm to the end and returning to
 the root shows that temporal immunity $\tau = 2m$ is enough to fully decontaminate $S$. A better bound may be obtained if we 
allow nonmonotonicity. This time we set $\tau = m$ and fully decontaminate each arm of the spider in turn and keep doing so until 
the entire spider is decontaminated. It is easy to verify that eventually (after possibly multiple rounds) this process ends.
 However one can obtain an even better estimate on $\iota(S)$.
\begin{thm}
Let $S$ be a spider on $n$ vertices such that the degree of 
the root is $\Delta$. If $m$ is the length of the longest arm of $S$ then $\iota(S) \le \Delta + \sqrt{\Delta^2 + 4m}$.
\label{thmSpider}
\end{thm} 
\begin{proof}
Arbitrarily order the arms of the spider $A_1, A_2, \ldots, A_{\Delta}$ and let the temporal immunity $\tau = t_0$, that is, the agent,
 when starting from the root, can decontaminate $t_0$ vertices on an arm before
 the exposed root gets recontaminated. Our strategy is going to be iterative and in each iteration, we are going to let the root
 get contaminated just once in the beginning, and after we decontaminate it, we will make sure that it does not get recontaminated 
during the course of that iteration. At the end of each iteration, $j$, we will have decontaminated all the arms of the spider 
from $A_1$ to $A_j$ including the root. Since this is going to be a nonmonotone strategy, parts or whole of these arms may be 
recontaminated during the course of the rest of the algorithm.

At the first iteration we start from the root, traverse $A_1$ to the end and return to the root. We proceed to decontaminate 
the rest of the spider using the following strategy. At the beginning of $j$th iteration, our agent is at the root of the spider
 and all the arms from $A_1, \ldots, A_{j-1}$ are fully decontaminated whereas $A_j, \ldots, l_\Delta$ are all fully contaminated 
(except for the root). The agent traverses each arm of the spider up to the farthest contaminated vertex and returns 
to the root in sequence starting from $A_j$ down to $A_1$. Note that, as mentioned before, we will allow the root to get contaminated
 just once in this iteration, that is, when our agent is traversing $A_j$. We want to fine tune our the temporal immunity $\tau$ 
such that 
once the agent returns after visiting all the vertices in $A_j$, during the rest of the iteration when the agent is visiting other
 arms, the root never gets contaminated again.
\begin{figure}[ht]
\begin{center}
\includegraphics[height=3in]{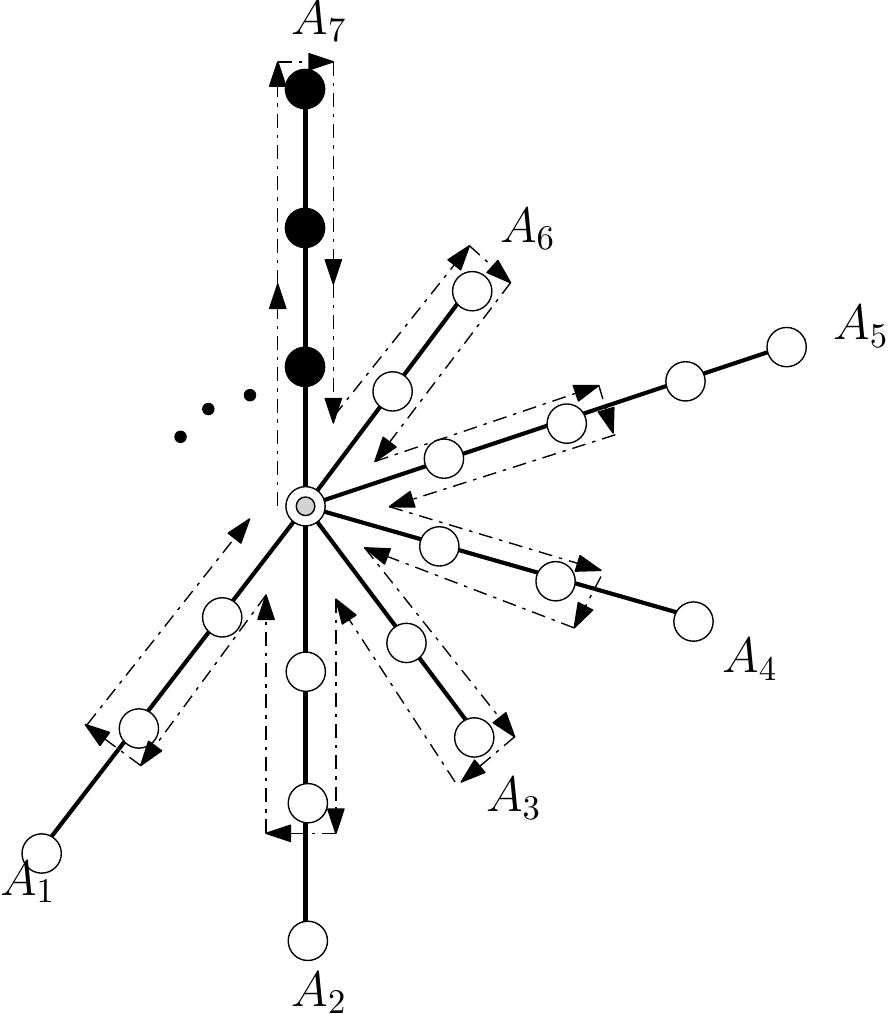}
\caption{{The agent is at the root. Arms $A_1,\ldots,A_6$ are decontaminated, represented as white dots. Dotted 
line segments show the path followed by the agent in $7$th iteration to decontaminate $A_7$.}}
\label{spider}
\end{center}
\end{figure}

Let $t_1$ be the total time needed to traverse the arms $A_j, \ldots, A_2$ after the root has been recontaminated (when the
 agent reached vertex $t_0$ of $A_j$). Then
\begin{equation}
t_1 < 2m + 2(j-1)\times \frac{t_0}{2}
\end{equation}
where the last term is the result of the constraint that the root may not be recontaminated in the current iteration. Now 
during the time $t_1$ at most $t_0/2$ vertices of $A_1$ should have been contaminated (once again to avoid recontamination 
of the root when we visit $A_1$). But that would have taken $t_0^2/2$ time units, therefore:
\begin{equation}\label{spiderBound}
t_1 = \frac{t_0^2}{2} \le 2m + 2(j-1)\times \frac{t_0}{2}. 
\end{equation}

Solving (\ref{spiderBound}) and using the fact that we get the worst bound at $j = \Delta$ we conclude that
$$
\tau = t_0 < \Delta + \sqrt{\Delta^2 + 4m}.
$$
\qed \end{proof}

\begin{corollary}
\label{corSpider}
If $S$ is a spider on $n$ vertices then $\iota(S) = O(\sqrt{n})$.
\end{corollary}
\begin{proof}
Let $S$ be rooted at a vertex $r$. If $deg(r) = \Delta\leq \sqrt{n}$, if follows from Theorem~\ref{thmSpider} that
$\iota(S)\leq \Delta + \sqrt{\Delta^2 + 4m} \leq 4\sqrt{n}$ which gives the claim. So, without loss of generality,
 assume that $\Delta> \sqrt{n}$.\\
Let $A_1, A_2,\ldots, A_\Delta$ be arms of $S$ with $|A_i|\leq |A_j|,$ for all $i\leq j$. Again without loss of generality
for some $k$\\
$$
|A_i| \begin{cases} <\sqrt{n} &\mbox{if } 1\leq i \leq k \\
\geq \sqrt{n} & \mbox{if } k<i\leq m \end{cases}
$$\\
Now consider a modified spider $S^* = S\setminus \bigcup_{1\leq i \leq k} A_i$ with $r$ as root. By pigeon hole principle,
$\Delta(S^*)\leq \sqrt{n}$. So we can apply technique used in proof of Theorem~\ref{thmSpider} 
to decontaminate $S^*$ with $\tau \leq 4\sqrt{n}$. Once $S^*$ is decontaminated, we use Lemma~\ref{lemmak-ary} to decontaminate
$(S\setminus S^*)\cup \{r\}$. Bound follows because height of this tree is less $\sqrt{n}$ and already decontaminated
$S$ never gets recontaminated by monotonicity in Lemma~\ref{lemmak-ary}.
\qed \end{proof}

\subsection{$k$-ary Tree}

\begin{lemma}\label{k-ary}
Any $k$-ary tree $T$ with height $h$ can be decontaminated with $\tau = 2h - 1$ using a monotone algorithm. 
\label{lemmak-ary}
\end{lemma}
\begin{proof}
First label the leaf vertices of $T$ so that $l_1, l_2, l_3, \ldots$ represents the order in which the leaves 
are visited if an in-order depth first traversal is performed on $T$ starting from the root vertex. Now it is 
straightforward to verify that if we start with the agent at the root, and visit each leaf in order $l_1, l_2, l_3, \dots$ 
returning to the root every time before visiting the next leaf, then $\tau = 2h - 1$ would be enough to decontaminate the 
entire $k$-ary tree. Also note that once decontaminated any leaf $l_i$ is never exposed again, and 
all nonleaf vertices, once decontaminated, are exposed for at most $2h-1$ time units.
Monotonicity follows.
\qed \end{proof}

Lemma \ref{k-ary} gives us the following corollary.
\begin{corollary}
Let $T$ be a perfect $k$-ary tree on $n$ vertices, then $\iota(T) = O(\log n)$.
\end{corollary}
In case of a binary tree the bound on temporal immunity can be slightly improved if we use the above strategy to first fully 
decontaminate the subtree rooted at the left child of the root, and then use the same method to decontaminate the 
subtree rooted at the right child of the root.
\begin{obs}
A binary tree with height $h$ can be decontaminated with an temporal immunity of $2h - 3$.
\end{obs}
\subsection{Mesh}
A $p\times q$ mesh is a graph that consists of $pq$ vertices. It is convenient to work with planar drawing of the graph where 
the vertices of $G=(V,E)$ are embedded on the integer coordinates of Cartesian plane. The vertices are named $v_{(i,j)}$ for
 $1 \leq i \leq q$, $1 \leq j \leq p$ corresponding to their coordinates in the planar embedding. There is an edge between
a pair of vertices if their euclidean distance is exactly $1$. We can partition $V$ into $C_1, C_2, \ldots, C_q$ 
column sets so that $v_{(i, j)} \in C_i$ 
for $1 \leq i \leq q$ and $1 \leq j \leq p$. Row sets of vertices $R_1, R_2, \ldots, R_p$ are defined analogously.

A simple approach to fully decontaminate a $p \times q$ mesh would be to place our agent at $v_{(1,1)}$ at $t = 0$, proceed to 
visit all vertices in the column till we reach $v_{(1, p)}$, move right one step to $v_{(2, p)}$ and proceed all the way 
down to $v_{(2, 1)}$. This process may now be continued by moving the agent to $v_{(3, 1)}$ and going on to decontaminate the 
entire graph column by column until we reach the last vertex. Clearly an temporal immunity of $2p - 1$ is enough for this strategy 
to monotonically decontaminate the entire graph. In \cite{daadaa2012network} the same strategy was used, albeit under 
a slightly different 
notion of \emph{temporal immunity}, to get a similar upper bound. However, once again we can improve this bound by resorting to a
 nonmonotone strategy. 

\begin{thm}
\label{meshUpperbound}
Let $G$ be a $p \times q$ mesh where $p \leq q$, then $\iota(G) \leq p$.
\end{thm}
\begin{proof}
We will describe a strategy to decontaminate $G$ in which we decontaminate each column nonmonotonically. However, once we declare
 a column to be decontaminated, we do not allow any of its vertices to be contaminated again.

Set the temporal immunity $\tau = p$ and start with the agent at $v_{(1,1)}$. Proceed all the way up to $v_{(1,p)}$, move the agent to 
the next column onto $v_{(2, p)}$, and then start traversing down the column until we reach $v_{(2,\lceil\frac{p}{2}\rceil+1)}$.
 Note that the vertices of $C_1$ had started getting recontaminated when the agent
 reached $v_{(2, p - 1)}$ because the exposure time of $v_{(1,1)}$ became equal to $\tau$ at that point. Now move the agent back 
to $C_1$ onto $v_{(1,\lceil\frac{p}{2}\rceil+1)}$ and proceed all the way down back to $v_{(1,1)}$. We declare that $C_1$ has been
 decontaminated and none of its vertices will be recontaminated during the course of decontamination of the rest of the graph. It
 is pertinent to note that at this point, $\Xi(v_{(2, p)}) = \tau - 1 = p - 1$.
\begin{figure}[ht]
\begin{center}
\includegraphics[height=3in]{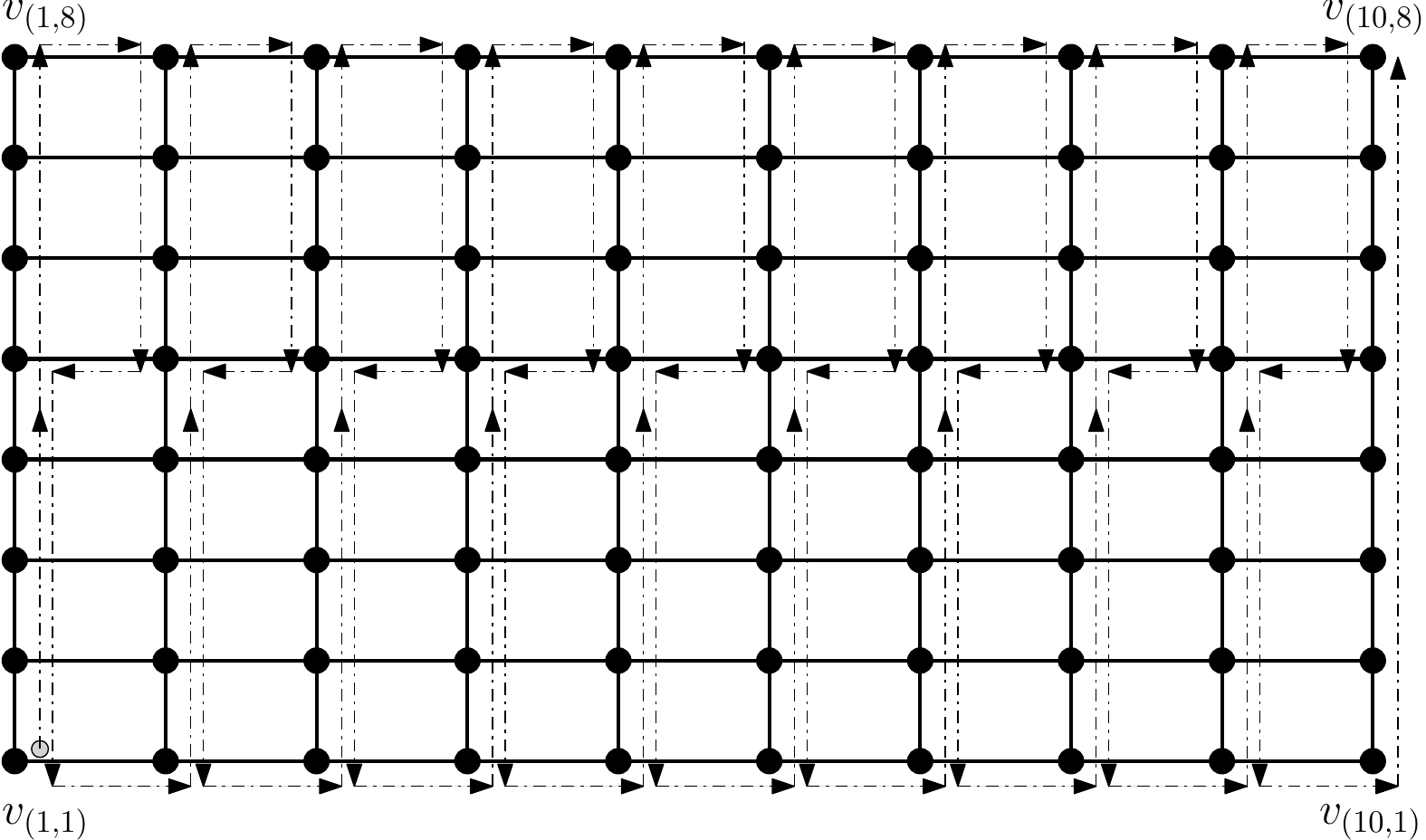}
\caption{Dotted line segments outline the path agent to decontaminate mesh graph with $\tau=8$. Once agent
returns to $v_{(1,1)}$, we declare first column decontaminated, and proceed to first vertex of second column, and henceforth.}
\label{mesh}
\end{center}
\end{figure}
To decontaminate the rest of the columns we use the following scheme. Assume that we have declared all the 
columns $C_1, C_2, \ldots, C_k$ to be decontaminated and our agent is at $v_{(k,1)}$. We also know that
 $\Xi(v_{(k+1, p)}) = \tau - 1$. We move the agent to the next column onto $v_{(k+1, 1)}$. At this point $v_{(k+1, p)}$
 becomes contaminated leaving $v_{(k, p)}$ exposed. We follow the same strategy as the one that we followed when we were
 decontaminating $C_1$. We move the agent all the way up to $v_{(k+1, p)}$, move to $C_{k+2}$, traverse all the way down 
to $v_{(k+2,\lceil\frac{p}{2}\rceil+1)}$, revert back to $C_{k+1}$ and move back down to $v_{(k+1, 1)}$ declaring column 
$C_{k+1}$ to be decontaminated. None of the vertices in $C_k$ will be recontaminated since $v_{(k, p)}$ had 
the maximum exposure time due to $v_{(k+1, p)}$, and we were able to 
decontaminate $v_{(k+1, p)}$ before $v_{(k, p)}$ got contaminated. Similarly, 
it is not difficult for the reader to verify that none of the rest of the vertices 
of $C_k$ are exposed long enough to be recontaminated.
\qed \end{proof}

\begin{corollary}
Let $G$ be a mesh on $n$ vertices, then $\iota(G) \leq \sqrt{n}$.
\end{corollary}
\begin{remark}
Strategy used in proof of Theorem~\ref{meshUpperbound} can also be used to decontaminate a {\em cylinder graph} (a mesh 
graph with an edge between the leftmost and the rightmost vertices on each row).
\end{remark}
In the following we present an asymptotically sharp lower bound for mesh graphs, but first we would like to
establish a graph isoperimetric result that we use in proof of lower bound.
\begin{lemma}
\label{lemmaMeshLowerbound}
 Let $G=(V,E)$ be an $\sqrt{n} \times \sqrt{n}$ mesh graph, then for any $W\subset V, |W|=\frac{n}{2}$,
size of maximum matching between $W$ and its complement has size at least $ \sqrt{n}$.
\end{lemma}
\begin{proof}
For ease of understanding let us say that a vertex is colored white if it is in set $W$, and black otherwise. An
 edge is monochromatic if both its endpoints have the same color, and nonmonochromatic otherwise.
Let $R_1, R_2, \ldots, R_{\sqrt{n}}$, and $C_1, C_2, \ldots, C_{\sqrt{n}}$
 be the row and column sets respectively. We observe following four possible cases:
\begin{case}
\label{case1}
 {\em For each row $R_i$, $0<|R_i\cap W|<\sqrt{n}$:}\\
Since $R_i$ contains vertices of both colors, it is clear that there will be at least one
nonmonochromatic edge. We pick one such edge from each $R_i$. As these edges are disjoint, we have a matching of size 
at least $\sqrt{n}$.
\end{case}
\begin{case}
 {\em There exist two rows $R_i, R_j$, such that $|R_i\cap W|=0$, and $|R_j\cap W|=\sqrt{n}$:}\\
We interchange the roles of rows and columns. Claim then follows from Case~\ref{case1}.
\end{case}
\begin{case}
\label{case3} 
{\em There exists a row $R_i$, such that $|R_i\cap W|=0$, and for every row $R_j$, \mbox{$|R_j\cap W|$}$\neq \sqrt{n}$:}\\
We present a scheme to match vertices in this case below. 

We will use two markers $\mathbf{b}$ (for bottom row), and $\mathbf{c}$ (for current row). 
In the beginning, both point to the first row of the mesh i.e $\mathbf{b}:=\mathbf{c}:=1$.

\begin{enumerate}
 \item Locate minimum $x\geq \mathbf{c}\geq \mathbf{b}$, such that $R_x\cap W = \emptyset$. 
If we can not find such an $x$, go to Step 3.
  \begin{enumerate}
  \item Now locate maximum $y<x\leq \mathbf{b}$, such that  $|R_y\cap W| \geq x-y+1$. For all the white vertices in $R_y$ we have
black vertices in corresponding columns of $R_x$. So for each such column there exists a pair of rows 
$R_i, R_{i+1}$ with a nonmonochromatic edge in that column where $y\leq i< x$. We set $\mathbf{c}:=x+1$, and call 
rows $R_j$ for $ y\leq j\leq x$ {\em matched}.
\item If we cannot find such a $y$, then we look for a minimum $z>x$, such that \mbox{$|R_z\cap W| \geq z-x+1$}. 
For all the white vertices in $R_z$, we can find nonmonochromatic edges as above. 
We set $\mathbf{c}:=z+1$ and $\mathbf{b}:=x+1$ and we call rows $R_j, x\leq j\leq z$ {\em matched}.
  \end{enumerate}
\item Repeat Step 1. Failure to find both $y$ and $z$ at any step would imply a contradiction 
because there are not enough black vertices as assumed. In worst case \\
$$\displaystyle \left|\bigcup_{i=1}^{x-1}R_i\right|\leq (x-1)\frac{\sqrt{n}}{2}$$ \\(alternating complete black and white rows),
and \\$$\displaystyle \left|\bigcup_{j=x+1}^{\sqrt{n}} R_j\right|< (\sqrt{n}-(x+1))\frac{\sqrt{n}}{2}$$
\item Match all unmatched rows as in Case 1.
\end{enumerate}

\end{case}
\begin{case}
 {\em There exists a row $R_i$, such that $|R_i\cap W|=\sqrt{n}$ and for every row $R_j$, $|R_j\cap W|\neq 0$:}\\
Claim in this case follows directly from the proof of Case~\ref{case3} by reversing the roles of $W$ and $\overline{W}$.
\end{case}
This concludes the proof of Lemma.
\qed
\end{proof}

Note that bound in Lemma~\ref{lemmaMeshLowerbound} is tight when $W$ is a {\em rectangular} subgrid. We do not know
of a tight example which is not rectangular in shape. We observe that since $\Delta=4$ for mesh,
Lemma~\ref{lemmaMeshLowerbound} also
 follows from vertex and edge isoperimetric inequalities proved in \cite{bollobas1991compressions}\cite{bollobas1991edge} except 
for a constant factor.

\begin{thm}
\label{thmMeshLowerbound}
Let $G$ be a $p \times q$ mesh where $p \leq q$, then $\iota(G) > \frac{p}{2}$.
\end{thm}
\proof
Let us assume the contrapositive i.e, a decontaminating algorithm exists with $\tau=\frac{p}{2}$.
For simplicity assume that $G$ is a $p \times p$ mesh and ignore the agent's moves in rest of the vertices if any. Let
$n=p^2$, then at some time step during this algorithm we will have exactly $\frac{n}{2}$ decontaminated vertices.
Lemma~\ref{lemmaMeshLowerbound} implies that at this stage at least $p$ vertices of $G$ are exposed through
at least $p$ disjoint edges to contaminated vertices. By considering all possible moves of the agent for next
$\frac{p}{2}$ steps it is clear that at least $\frac{p}{2}$ vertices will be recontaminated, and no matter what the agent does
this can decontaminate at most $\frac{p}{2}$ vertices making no progress at all. It already gives that number of
decontaminated vertices can never exceed $\frac{n}{2}+\frac{p}{2}$. It follows that no decontaminating algorithm 
exists with assumed temporal immunity.
\qed
\section{General Trees}
\label{secTrees}
To upper bound $\iota$ for general trees, we will try to adapt the strategy used 
 to decontaminate $k$-ary trees. The simplest approach is to naively
 apply the same strategy on a given tree $T$ as before, this time considering the center vertex (choose one arbitrarily if there
 are two center vertices) of the tree to be the \emph{root} and then visiting each of the leaves of
 $T$ in the depth first search discovery order, every time returning to the center vertex, as in the previous case. It is clear 
that an temporal immunity $\tau = 2\cdot rad(T) = diam(T) + 1$ is sufficient to fully decontaminate $T$ but 
the diameter of a tree on $n$ 
vertices can easily be $O(n)$. However, we can use nonmonotonicity to our advantage by letting a controlled number of 
vertices get recontaminated so that we get a much stronger bound even for trees with large diameters. 

We will need the following lemma which describes a monotone strategy to decontaminate trees with small height.

\begin{lemma}
\label{smalltreelemma}
Any rooted tree $T$, with height $h$, can be decontaminated with temporal immunity $\tau\geq \alpha h$, in time $cn$, 
where $n$ is the number of vertices in $T$, and $c\leq 4\frac{\alpha-1}{\alpha-2}$ for any positive 
$\alpha>2$. 
\end{lemma}
\begin{proof}
Assuming an arbitrary tree with height and temporal immunity $\tau$ as above, we present an algorithm with claimed time complexity.\\

Let $l$ be maximum integer such that there exists a subtree $P_1$ rooted at $p_1$ with $|P_1|>h\frac{\alpha}{2}-h$ at level $l$, and let 
$s_1, s_2,\ldots, s_m$ be children of one such $s$; for all $i$ let $S_i$ be subtrees rooted at
 $s_i$, then $|S_i|<h\frac{\alpha}{2}-h$ by maximality of $l$. Now let
$j$ be the largest integer such that $|S_1\cup \cdots \cup S_j \cup \{p_1\}|<(h\frac{\alpha}{2}-h)$, we define
$X_1:=S_1\cup \cdots \cup S_j \cup \{p_1\}$.
We similarly define $X_2,\cdots, X_k$,
as maximal subtrees all rooted at $p_1$ making sure that $\frac{1}{2}h(\frac{\alpha}{2}-1)<|X_i|<h(\frac{\alpha}{2}-1)$, 
with possible exception of
$X_k$ which might be of smaller cardinality.\\
\begin{figure}[ht]
\begin{center}
\includegraphics[height=2.5in,width=5in]{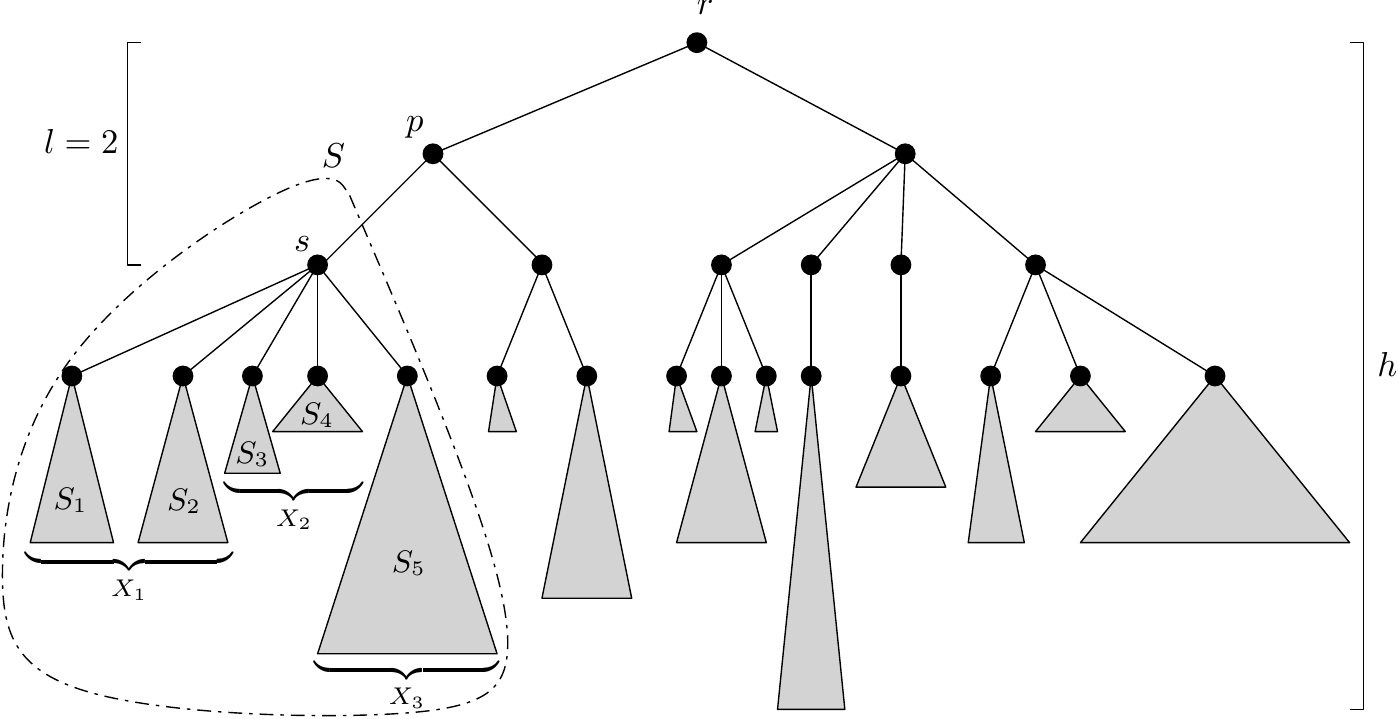}
\caption{Example illustrating grouping of $S_i$ into $X_j$}
\label{TreeLemma}
\end{center}
\end{figure}

For decontamination process,
the agent starts at root of $T$, walks its way to $p_1$, performs a depth first search traversal 
on each $X_i$ one by one. We can afford this because immunity is
strictly greater than the amount of time it takes to perform the traversal on each $X_i$; in fact
 its easy to see that any $\tau\geq h{\alpha}-2h$ is enough to completely clean $P_1$.\\

Next step is to walk up to $p_2$ parent of $p_1$. The plan is to make sure that $p_2$ never gets recontaminated. Let
$P_2$ be the subtree rooted at $p_2$.
Arbitrarily choose any subtree $R\subseteq P_2\setminus P_1$, at minimum possible distance from $P_2$ (e.g. potentially $R=P_2$),
with the property that for all subtrees $R_i$ of $R$, $|R_i|<h\frac{\alpha}{2}-h$ as before. 
We will group $R_i$'s into $X_j$'s as before but this time
after performing depth first search traversal on each $X_j$, we will pay a visit to $p_2$, making sure it remains decontaminated.
Once $P_2$ is decontaminated, we proceed to $p_3$ the parent of $p_2$ and repeat the process until $p_j$ is the root of $T$, 
and that we are done with decontamination
process. From the fact that each $X_i$ is small enough, it is easy to see that $\tau=\alpha h$ 
is enough for the process. We can always
 group any tree into at most $ 2\frac{n}{q}$ subtrees each of size $(h\frac{\alpha}{2}-h)\leq |X_i| \leq h({\alpha}-2)$
where $q=h({\alpha}-2)$. 
The agent spends at most $2q$
time units on depth first traversal and $2h$ time units on visiting some $p_j$ potentially at distance $h$ for each such subtree. 
Total amount of time spent in the process 
is
\begin{align*}
 &\leq 2\frac{n}{q}\times 2(q+h)\\
&\leq 2\frac{n}{q}\times 2(q+\frac{q}{\alpha-2})\\
&= 2{n}\times 2(1+\frac{1}{\alpha-2})\\
&= 4{n}\times \frac{\alpha-1}{\alpha-2}\\
\end{align*}
\qed \end{proof}

\begin{thm}
\label{thmTreesUpperBound}
Let $T$ be a tree on $n$ vertices, then $\iota(T) = O(\sqrt{n})$.
\end{thm}

\begin{proof}
First of all, following is easily seen:
\begin{obs}  
\label{observationMonotone}
Decontamination strategy in Lemma~\ref{smalltreelemma} is a monotone strategy.
\end{obs}

Now let $c$ be a center vertex of $T$ and let $m$ be the number of leaves in $T$. Recall that an arm $A_i$ is
a set of vertices that lie on the path from $c$ to a leaf $l_i$ for all $1\leq i \leq m$. Given a tree $T$ rooted at $v$,
we denote by $T_{x}(v)$ a subtree of $T$ that is attained by removing all vertices from $T$ that 
are at distance more than $x$ from $v$ i.e, $T_{x}$ is $T$ truncated at depth $x$.
Assume without loss of generality that leaves $l_i$ are sorted in their depth first search discovery ordering. This implies
an ordering on arms $A_i$.
Note that $A_i\setminus \{c\}$ are not disjoint in general.
 Once we have an order, agent will start decontaminating 
 arms one by one according to following algorithm.
 
 \begin{itemize}
  \item For $i=1$ to $m$
  \begin{itemize}
    \item Perform an auxiliary step and apply Lemma~\ref{smalltreelemma} on $T_{\sqrt{n}}(c)$ with $\alpha=3$.
    \item Move the agent from $c$ towards leaf $l_i$ until it reaches a vertex $v_j$ with $deg(v_j)> 2$. 
    We will apply Lemma~\ref{smalltreelemma} on $T_{10\sqrt{n}}(v_j)$ again with $\alpha=3$. 
    After performing an auxiliary decontamination step, we will not perform any more auxiliary steps
    for next $5\sqrt{n}$ time units of this walk. Since $l_i$ can be at distance at most $\frac{n}{2}$ from $c$,
    total number of auxiliary steps we perform on this walk is bounded from above by $\frac{n}{10\sqrt{n}}$. It also
    follows that no vertex lies in more than two $T_{10\sqrt{n}}(v_j)$'s. 
    We return to $c$ along the shortest path. 
  \end{itemize}
 \end{itemize}

To analyze this scheme, we find following definition useful:
\begin{definition}
 A vertex $v$ in tree $T$ is called secured at some time step $i$ if it never gets contaminated again. 
\end{definition}

Agent decontaminates a new arm $A_i$ in $i$th iteration of the algorithm. We Observe that 
 \begin{claim}
Following invariants hold for every step of the algorithm:
\begin{enumerate}
 \item[(i)] Root $c$ is secured at iteration 1.
 \item[(ii)] For any secured vertex $v$, and a contaminated vertex $w$, which is in same branch as
 $A_i$, $dist(v,w)>\sqrt{n}$ at start of iteration $i+1$.
 \item[(iii)] All vertices $v_j\in A_i$ are secured at start of iteration $i+1$.
\end{enumerate}
\end{claim}

\begin{proof}
We fix $\tau=30\sqrt{n}$. Let $\Gamma(i)$ be the time spent in the algorithm at iteration $i$, 
then $\Gamma(i)$ can be broken down into three parts: 
$(1)$ the time spent performing auxiliary decontamination at $c$, 
$(2)$ the time spent visiting $l_i$, and
$(3)$ time spent at each auxiliary step on the way to $l_i$, which is $8a_j$ with $a_j$ be size of the tree used in auxiliary step.
We have,
\begin{align*}
 \Gamma(i) &\leq 8n + n + 8\Sigma_{j} a_j   \\
  &\leq 8n + n + 16 n   \\
 &= 25n 
\end{align*}
where we use the fact that $\Sigma_{j} a_j$ cannot be more than twice the number of total vertices since each vertex is used 
in at most two such auxiliary steps. Since $\tau=30\sqrt{n}$, after performing an auxiliary decontamination step
 on tree with $\sqrt{n}$ height, it takes $\geq 30n$ for the contamination 
to creep back to the root which is less than the time spent in one iteration. This fact along with 
Observation~\ref{observationMonotone} gives the first invariant (i).\\

Now a vertex $v$ is secured only if $v\in A_j$ for some $j\leq i$. If $v$ lies in a different branch than the one
$A_i$ lies in then invariant (ii), and (iii) follow from (i) i.e. 
if $c$ is secured then contamination has no way to spread from one branch to
another, and if a branch has been decontaminated, it will not get recontaminated. 
For any contaminated vertex $w$, $dist(c,w)>\sqrt{n}$ implies that $dist(v,w)>\sqrt{n}$, for any $v$ in a fully
decontaminated branch. So we assume without loss of generality that $v$ lies in the same branch as $A_i$.
From order defined on leaves in which we decontaminate them, it is clear that after iteration $i$, 
closest secured vertex to any
contaminated vertex, lies in arm $A_i$. 
A direct consequence of performing auxiliary decontamination during iteration $i$ is that any contaminated vertex
is at distance more than $5\sqrt{n}$ from closest $v\in A_i$. When we have completed iteration $i$, it is still 
more than $4\sqrt{n}$ distance away. Which implies part invariant $(ii)$.\\

For any contaminated vertex $w$, any $u\in A_i$, and any $v\in A_j$ for $j<i$ all contained in the same branch 
it holds that $dist(u,w) < dist(v,w)$. It follows that  $v\in A_j$ for $ j< i$
never get contaminated during decontamination process of their branch. This along with $(i)$ implies $(iii)$. 
\qed \end{proof}

Claim completes the proof of Theorem with $\iota=30\sqrt{n}$. Although constant can be improved upto $6$, 
but resulting structure
of proof is messier and, in our opinion, doesn't yield any further insight into the problem. 
\qed \end{proof}

\section{Discussion}
\label{secDiscussion}
While we presented some interesting results, we would like to mention that there are still some very basic 
questions that seem to be open for further investigations.
For example, we showed that for any tree $T$, $\iota(T)=O(\sqrt{n})$, yet it is not clear whether this is 
asymptotically optimal or not. Using somewhat involved argument, it can be shown that there exist trees $T$
on $n$ vertices for which
$\iota(T)=\Omega(n^{\frac{1}{3}+\epsilon})$ for any constant $\epsilon>0$.
Its also noteworthy to state that if we limit algorithms to be monotone, its easy to see that
$\iota(T)=\Theta(n)$ e.g. consider a spider with three arms of equal length.\\
Another interesting topology is that of planar graphs. Since mesh is a planar graph, 
it directly follows from Theorem~\ref{thmMeshLowerbound} that
\begin{corollary}
\label{corPlanarLowerBound}
There exist planar graphs on $n$ vertices such that their immunity number $\iota > \frac{\sqrt{n}}{2}$.
\end{corollary} 
We believe that
\begin{conjecture}
 Any planar graph $G$ on $n$ vertices can be decontaminated with $\tau(G)=O(\sqrt{n})$.
\end{conjecture}
A similar bound for a slightly different problem of bounding, $s(G)$, lends some credence to the above conjecture.
Search number, $s(G)$, is the minimum number of agents needed to decontaminate a graph with $\tau=0$. Following was proved in 
\cite{alonchasing} by Alon et al., but we note that proof we present here is simpler, shorter, and more intuitive.
\begin{thm}
\label{thmPlanarGraphUpper}
Any planar graph $G=(V,E)$ on $n$ vertices can be decontaminated with $s(G)=O(\sqrt{n})$ agents where 
vertices of $G$ don't have any immunity.
\end{thm}
\begin{proof}
We partition $V$ into three sets $V_1, V_2,$ and $S$ using Planar Separator Theorem \cite{lipton1979separator}, 
where $|V_i|\leq \frac{2n}{3}$, 
$|S|\leq 3\sqrt{n}$, owing to improvements in \cite{djidjev1982problem}, and 
for any $v\in V_1$, and any $w\in V_2$, edge $vw\notin E$. We place $3\sqrt{n}$ agents
on $S$ to make sure that contamination can not spread from $V_1$ to $V_2$, or vice versa. Let
$G_1=(V_1,E_1)$ be the subgraph of $G$ where an edge of $E$ is in $E_1$ if both its endpoints are in $V_1$. Similarly,
define $G_2=(V_2,E_2)$. Now lets say it takes 
$s(G_1)$ agents to decontaminate $G_1$, once $G_1$ is fully decontaminated, we can reuse all those agents to 
decontaminate $G_2$. Since both $G_1$ and $G_2$ are also planar graph, this gives us an obvious recurrence for $s(G)$:
$$
s(G) \le \max(s(G_1),s(G_2)) + 3\sqrt{n}
$$
So the total number of agents required is at most $3\sqrt{n}+3\sqrt{\frac{2n}{3}}+\ldots=O(\sqrt{n})$ which completes the proof.
\qed \end{proof}

Technique used in proof of Theorem~\ref{thmPlanarGraphUpper} may help devise a similar proof for the conjectured bound on immunity
number of planar graph. 
In any case, we do have a hunch that
Planar Separator Theorem may be beneficial in that case as well.\\

Also its not hard to show that $K_n$ has highest immunity number among all graphs on $n$ vertices.
\begin{thm}
\label{thmGraphUpperBound}
Any connected graph $G=(V,E)$ on $n$ vertices can be decontaminated with $\tau=n-1$.
\end{thm}
\begin{proof}
Start with a agent on arbitrary vertex $v_1$, and at each time step 
keep walking the agent to successive nonvisited neighbors. If we exhaust all $V$ then we are done since 
we visited all vertices before first vertex got recontaminated. Otherwise agent gets stuck at the end of some 
path $v_1, \ldots, v_{k-1}, v_k$ such that all neighbors of $v_k$ have already been 
visited. We call such a vertex {\em terminal} vertex. For the rest of decontamination process, we will assume that $v_k$ does not exist. 
We traverse the agent back along $v_k, v_{k-1},\ldots, v_1$ to reach $v_1$, and then come back along same path to reach $v_{k-1}$.
This time the agent moves to some other neighbor of $v_{k-1}$ if any, 
and continue as before either finding another another terminal vertex and deleting it too or finding a 
cycle on rest of the vertices. In either case, process completes in finite time. Since the agent decontaminated terminal vertices,
they cannot contaminate any other vertex after they have been visited. And since, every time the agent encounters a terminal vertex
it goes
back to $v_1$, and visits all its neighbors (all of which lie on agent's path back to $v_1$)
in the next less than $n-1$ steps, terminal vertices cannot get contaminated again.
Vertices that are not terminal are decontaminated at the end of the process because they are visited in the traversal on cycle
which takes at most $n-1$ steps after we leave $v_1$. The claim follows.
\qed \end{proof}
This might tempt one to conjecture that $\iota(G)$ is an increasing graph property i.e. if we add new edges to $G$ then
immunity number can only go up. But as the following claim illustrates, that is not the case. 
\begin{obs}
\label{obsKahn}
Immunity number is not an increasing graph property. \cite{Kahn}
\end{obs}
\begin{proof}
Consider the following counter-example.
Let $G$ be a spider with $2\sqrt{n}$ arms labeled $A_1,A_2,\ldots, A_{2\sqrt{n}}$, where\\
$$
|A_i| = \begin{cases} \sqrt{n}-1 &\mbox{if } i \equiv 0 \\
1 & \mbox{if } i \equiv 1 \end{cases} \pmod{2}.
$$
\\
\begin{figure}[h]
\begin{center}
\includegraphics[height=2in]{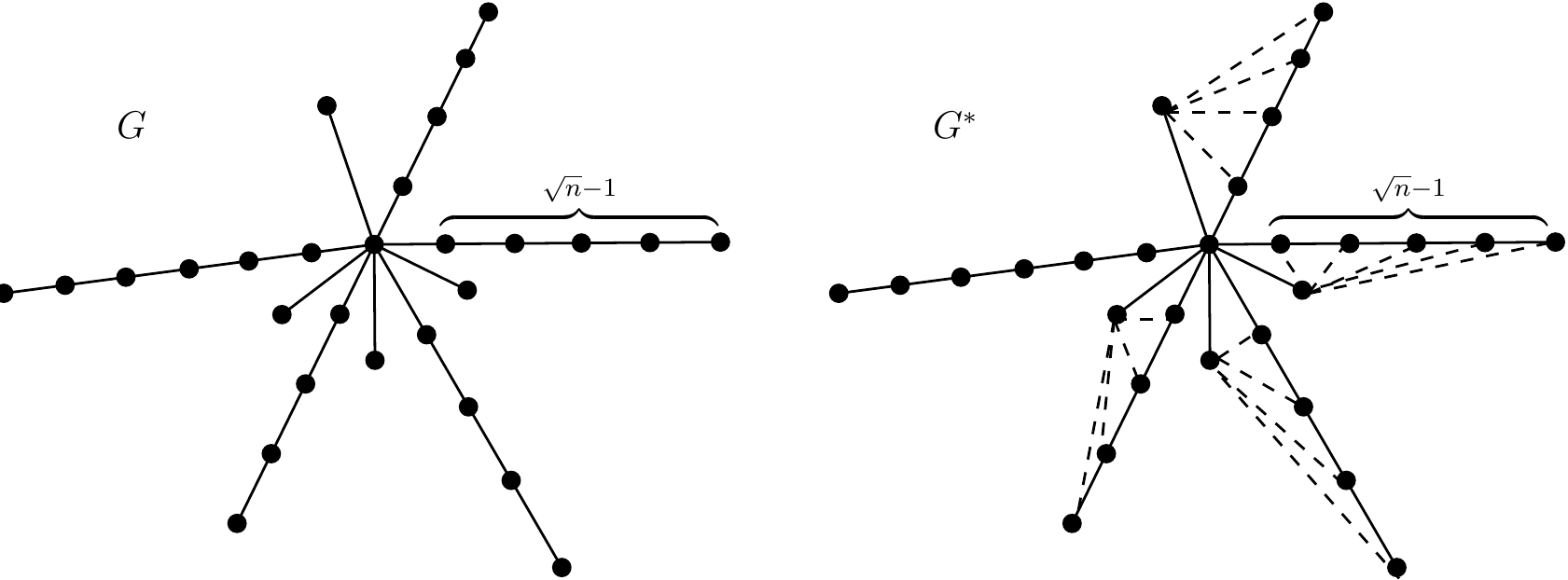}
\caption{(Left) $G$ is a spider tree, and can't be decontaminated with small temporal immunity. We get $G^*$ (on right)
by adding dashed edges, and its easy to see that we can decontaminate $G^*$ with $\tau=2$.
}
\label{nonmonotonicity}
\end{center}
\end{figure}
 Now construct $G^*$ by adding edges
 $vw$ where $ v\in A_i, w\in A_{i+1},$ for all $i \equiv 1 \pmod{2}$ then we can decontaminate $G^*$ with $\tau(G^*)=2$. 
 We leave it as an exercise for the reader to verify that $\iota(G) > 2$.\\ 
\qed \end{proof}
There are natural generalizations of the problem investigated in this paper to directed and weighted graphs. One can also look at
behavior of immunity number of random graphs.
\section{Acknowledgments}
We are grateful to Devendra Desai, Jeff Kahn, and William Steiger for their insightful comments and discussions. 
\bibliography{decon}
\bibliographystyle{plain}
\newpage

\end{document}